\documentclass[12pt,a4paper,reqno]{amsart}
\usepackage[active]{srcltx}
\usepackage{enumerate}
\usepackage{amssymb}
\usepackage[usenames]{xcolor}
\usepackage{url}

\usepackage{cases}

\newcommand{\Z}{{\mathbb Z}}

\newcommand{\Hom}{\operatorname{Hom}\nolimits}

\newcommand{\Aut}{\operatorname{Aut}\nolimits}

\newcommand{\aut}{\operatorname{aut}\nolimits}
\newcommand{\Out}{\operatorname{Out}\nolimits}

\newcommand{\Syl}{\operatorname{Syl}\nolimits}

\newcommand{\A}{\ifmmode{\mathcal{A}}\else${\mathcal{A}}$\fi}
\newcommand{\K}{\ifmmode{\mathcal{K}}\else${\mathcal{K}}$\fi}
\newcommand{\U}{\ifmmode{\mathcal{U}}\else${\mathcal{U}}$\fi}
\newcommand{\M}{\ifmmode{\mathcal{M}}\else${\mathcal{M}}$\fi}
\newcommand{\N}{\ifmmode{\mathcal{N}}\else${\mathcal{M}}$\fi}
\newcommand{\Ff}{\ifmmode{\mathcal{F}}\else${\mathcal{F}}$\fi}
\newcommand{\Ll}{\ifmmode{\mathcal{L}}\else${\mathcal{L}}$\fi}
\newcommand{\T}{\ifmmode{\mathcal{T}}\else${\mathcal{T}}$\fi}

\newtheorem{theorem}{Theorem}[section]
\newtheorem*{theorem*}{Theorem}
\newtheorem{proposition}[theorem]{Proposition}

\newtheorem{lemma}[theorem]{Lemma}

\theoremstyle{definition}
\newtheorem{definition}[theorem]{Definition}

\newtheorem{example}[theorem]{Example}

\theoremstyle{remark}

\theoremstyle{plain}

\title[Thompson's theorems for fusion systems]%
{On Thompson's $p$-complement theorems for saturated fusion systems}

\author{Jon Gonz\'alez-S\'anchez}
\address{Departamento de Matem\'aticas,%
Universidad de Pa{\'\i}s Vasco,
Apartado 644,
48080 Bilbao
Spain.}
\email{jon.gonzalez@ehu.es}

\author{Albert Ruiz}
\address{Departament de Matem{\`a}tiques,
Universitat Aut{\`o}noma de Barcelona, 08193 Cerdanyola del
Vall{\`e}s, Spain.}
\email{Albert.Ruiz@uab.cat}

\author{Antonio Viruel}
\address{Departamento de {\'A}lgebra, Geometr{\'\i}a y Topolog{\'\i}a,
Universidad de M{\'a}\-la\-ga, Campus de Teatinos, s/n, 29071-M{\'a}laga,
Spain.}
\email{viruel@agt.cie.uma.es}

\thanks{
First author is partially supported by MINECO MTM2011-28229-C02-01. Second author
is partially supported by FEDER-MICINN grant MTM2010-20692. Third author is partially
supported by FEDER-MICINN grant MTM2010-18089, and
Junta de Andaluc{\'\i}a grants FQM-213 and P07-FQM-2863. Second and third authors
are partially supported by Generalitat de Catalunya grant
2009SGR-1092.
}

\begin{document}

\begin{abstract}
In this short note we prove that a saturated fusion system admitting some special type of automorphism is nilpotent. This generalizes classical results by J.G.\ Thompson.
\end{abstract}

\maketitle


\section{Introduction}

In his PhD\ Thesis, John\ G.\ Thompson proved the long standing conjecture that the Frobenius kernel of a
Frobenius group is nilpotent \cite{Thompson1}. The Frobenius kernel always admits a fixed-point-free
automorphism of prime order which turned out to be a sufficient condition for the nilpotency.
In fact both results, the nilpotency of the Frobenius kernel and the nilpotency of a finite
group admitting an fixed-point-free automorphism of prime order are equivalent.
In order to prove this result, Thompson introduced his famous $p$-nilpotency criterion for
an odd prime $p$, a group $G$ is $p$-nilpotent if and only if the normalizer of the $J$ subgroup
of a Sylow $p$-subgroup and the centralizer of the center of a Sylow $p$-subgroup are $p$-nilpotent
(see \cite{Thompson3}). In fact, Thompson proved that a group $G$ admitting automorphisms that
leave invariant some special subgroups is $p$-nilpotent. The following translates \cite[Theorem A]{Thompson2} for saturated fusion systems, and includes an extra hypothesis to cover the $p=2$ case (see Definition \ref{def:h-free}):

\begin{theorem}\label{thm:A}
Let $\Ff$ be a saturated fusion system over a $p$-group $S$ such that either $p$ is odd, or $p=2$ and $\Ff$ is $\Sigma_4$-free.
Let $\U$ be a group of automorphisms of $(S,\Ff)$. Suppose for every $\U$-invariant normal subgroup $Q\unlhd S$, $\Aut_\Ff(Q)$ is a $p$-group. Then $\Ff=\Ff_S(S)$.
\end{theorem}

As a consequence of \cite[Theorem A]{Thompson2} Thompson proved that a group $G$ admitting a fixed-point-free
automorphism of prime order is $p$-nilpotent for all primes, and therefore the group is nilpotent \cite[Theorem 1]{Thompson1}.

Given a finite group $G$ and $\phi$ a prime order automorphism without fix points, then $\phi$ fixes
a  Sylow $p$-subgroup $S$ of $G$. For a fixed subgroup $H$ one has that $N_G(H)$ and $C_G(H)$ are also fixed.
Therefore, $\phi$ acts on the quotient group $N_G(H)/C_G(H)$ and the action on this quotient is without
fix points (see Lemma  \ref{thm:gor_10.1.2_3}(b)). So one can translate the concept of a
fixed-point-free automorphism of prime order from the category of finite groups to the category
of saturated fusion systems (see Definition \ref{def:fixpointfreefusion} and Proposition \ref{prop:fixedpointfree}).
It turned out that the concept of fixed-point-free automorphism in the category of saturated fusion systems
is more general since there exist finite groups admitting automorphism of prime order
with fix points such that the induced automorphism on the fusion category is fixed-point-free (see Example \ref{Ex:gp-vs-fusion}). Nevertheless, $p$-nilpotency holds under this mild assumption, as we prove this generalization of \cite[Theorem 1]{Thompson1}:

\begin{theorem}\label{thm:B}
Let $\Ff$ be a saturated fusion system over a $p$ group $S$ such that either $p$ is odd, or $p=2$ and $\Ff$ is $\Sigma_4$-free.
Then, if $(S,\Ff)$ admits a prime order fixed-point-free automorphism, $\Ff=\Ff_S(S)$.
\end{theorem}

These results contribute to the list of nilpotency criteria for saturated fusion systems that generalize
classical criteria for finite groups, and fits within the framework of previous work by
Kessar-Linckelmann \cite{KL}, D{\'\i}az-Glesser-Mazza-Park \cite{DGMP},
D{\'\i}az-Glesser-Park-Stancu \cite{DGPS}, Cantarero-Scherer-Viruel \cite{CSV} and Craven \cite{Craven}.
Indeed, Theorem \ref{thm:A} can also be deduced from \cite[Corollary 4.6]{DGMP}, although the proof of Thompson's $p$-nilpotence criterion in \cite{DGMP} resorts to the group case, while the proof presented here is purely fusion theoretical. Another independent fusion theoretical proof of the odd prime case in Theorem \ref{thm:B} can be found in \cite{Craven}

This note is organized as follows. In Section \ref{Definitions} we recall the main properties of saturated fusion system
and we introduce the concept of a fixed-point-free automorphism of a saturated fusion system.
In Section \ref{Thompson} we provide a unified proof of Theorems \ref{thm:A} and \ref{thm:B}.

\section{Background on saturated fusion systems and finite groups} \label{Definitions}

In this section we review the concept of a
saturated fusion system over a $p$-group $S$ as defined in \cite{BLO2}, and define fixed-point-free automorphism of a saturated fusion system.

\begin{definition}\label{def:fusion_system}
A fusion system $\Ff$ over a finite $p$-group $S$ is a category
whose objects are the subgroups of $S$, and whose morphisms sets
$\Hom_\Ff(P,Q)$ satisfy the following two conditions:
\begin{enumerate}[(a)]
\item $\Hom_S(P,Q) \subseteq \Hom_\Ff(P,Q) \subseteq
\operatorname{Inj}(P,Q)$ for all $P$ and $Q$ subgroups of $S$.

\item Every morphism in $\Ff$ factors as an isomorphism in $\Ff$
followed by an inclusion.
\end{enumerate}
\end{definition}

We say that two subgroups $P$,$Q \leq S$ are
\textit{$\Ff$-conjugate} if there is an isomorphism between them
in $\Ff$. As all the morphisms are injective by condition (b), we
denote by $\Aut_\Ff(P)$ the group $\Hom_\Ff(P,P)$. We denote by
$\Out_\Ff(P)$ the quotient group $\Aut_\Ff(P)/\Aut_P(P)$.

The fusion systems that we consider are saturated, so we need the following
definitions:

\begin{definition}\label{def:N}
Let $\Ff$ be a fusion system over a $p$-group $S$.
\begin{itemize}
\item A subgroup $P \leq S$ is \emph{fully centralized in $\Ff$}
if $|C_S(P)|\geq|C_S(P')|$ for all $P'$ which is $\Ff$-conjugate
to $P$.
\item A subgroup $P \leq S$ is \emph{fully normalized in
$\Ff$} if $|N_S(P)|\geq|N_S(P')|$ for all $P'$ which is
$\Ff$-conjugate to $P$.
\item $\Ff$ is a \emph{saturated fusion
system} if the following two conditions hold:
\begin{enumerate}[(I)]
\item Every fully normalized in $\Ff$ subgroup $P \leq S$ is fully
centralized in $\Ff$ and
$\Aut_S(P) \in \Syl_p(\Aut_\Ff(P))$. \item If
$P\leq S$ and $\varphi \in \Hom_\Ff(P,S)$ are such that $\varphi
P$ is fully centralized, and if we set
$$ N_\varphi=\{ g \in N_S(P) \mid \varphi c_g \varphi^{-1} \in
\Aut_S(\varphi P)\} ,
$$ then there is $\overline\varphi \in \Hom_\Ff(N_\varphi,S)$ such
that $\overline\varphi|_P=\varphi$.
\end{enumerate}
\end{itemize}
\end{definition}

As expected, every finite group $G$ gives rise to a saturated
fusion system over $S$, a Sylow $p$-subgroup of $G$, denoted by $(S,\Ff_G(S))$ \cite[Proposition 1.3]{BLO2}.
But there exist saturated fusion systems that are not the fusion system of any finite group, e.g.\ \cite[Section 9]{BLO2} or \cite{RV}.

Let $\Ff$ be a fusion system over a $p$-group $S$, and $Q$ a subgroup of $S$.
We can take the \emph{normalizer of $Q$ in $\Ff$} as the fusion system
over the normalizer of $Q$ in $S$, $N_S(Q)$, with morphisms:
$$
\Hom_{N_\Ff(Q)}(P,P')=\{\varphi\in\Hom_\Ff(P,P') \mid
 \exists \psi\in\Hom_\Ff(PQ,P'Q), \psi|_P=\varphi\}.
$$
Although, $(N_S(Q),N_\Ff(Q))$ is not always a saturated fusion system, it is so when $Q$ is fully normalized \cite[Proposition A.6]{BLO2}:
\begin{proposition}\label{prop:fullynormalized}
Let $\Ff$ be a saturated fusion system over a $p$-group $S$. If $Q\leq S$ is fully normalized in $\Ff$, then $(N_S(Q),N_\Ff(Q))$ is a saturated fusion system.
\end{proposition}

We also need results concerning the quotients of saturated fusion systems. Recall
that if $(S,\Ff)$ is a saturated fusion system, we say that $Q\leq S$ is a
\textit{weakly $\Ff$-closed} subgroup if $Q$ is not $\Ff$-conjugate to any other
subgroup of $S$.

In \cite[Lemma 2.6]{Oliver-p-odd} we can find the following result:

\begin{lemma}\label{lemma:oliver-quotients}
Let $(S,\Ff)$ be a saturated fusion system, and let $Q\vartriangleleft S$ be a
weakly $\Ff$-closed subgroup. Let $\Ff/Q$ be the fusion system over $S/Q$ defined by
setting:
$$
\Hom_{\Ff/Q}(P/Q,P'/Q)=\{\varphi/Q \mid \varphi\in\Hom_\Ff(P,P')\}
$$
for all $P$, $P'\leq S$ which contains $Q$. Then $\Ff/Q$ is saturated.
\end{lemma}

Some classical
results for finite groups can be generalized to saturated fusion
systems, as for example, Alperin's Fusion Theorem for saturated fusion systems \cite[Theorem
A.10]{BLO2}:

\begin{definition}\label{def:centric-radical}
Let $\Ff$ be any fusion system over a $p$-group $S$. A subgroup
$P\leq S$ is:
\begin{itemize}
\item  \emph{$\Ff$-centric} if $P$ and all its $\Ff$-conjugates
contain their $S$-centralizers.
\item \emph{$\Ff$-radical} if $\Out_\Ff(P)$ is
$p$-reduced, that is, if $\Out_\Ff(P)$ has no nontrivial normal $p$-subgroups.
\end{itemize}
\end{definition}

\begin{theorem}\label{teo:alperin}
Let $\Ff$ be a saturated fusion system over $S$. Then for each
morphism $\psi\in\Aut_\Ff(P,P')$, there exists a sequence of
subgroups of $S$
$$P=P_0, P_1,\ldots,P_k=P'\quad\text{and}\quad Q_1,Q_2,\ldots,Q_k,$$
and morphisms $\psi_i\in\Aut_\Ff(Q_i)$, such that
\begin{itemize}
\item $Q_i$ is fully normalized in $\Ff$, $\Ff$-radical and
  $\Ff$-centric for each $i$;
\item $P_{i-1},P_i\leq Q_i$ and $\psi_i(P_{i-1})=P_i$ for each
$i$;
  and
\item $\psi=\psi_k\circ\psi_{k-1}\circ\cdots\circ\psi_1$.
\end{itemize}
\end{theorem}

Here we will recall some definitions and results concerning groups \cite[Chapter 2]{gorenstein-68}
and  group automorphisms \cite[Chapter 10]{gorenstein-68}:

A finite group $G$ is \textit{nilpotent} if the lower central series of $G$, defined as
$$
\gamma_1(G)=G, \, \gamma_i(G)=[\gamma_{i-1}(G),G] \mbox{ for $i\geq2$}
$$
satisfies that there exists $m$ such that $\gamma_m=\{1\}$.

As examples we have that every finite $p$-group is nilpotent \cite[Theorem 3.3(iii)]{gorenstein-68},
and more generally:
\begin{theorem}
A finite group $G$ is nilpotent if and only if it is the direct product of its Sylow $p$-subgroups.
\end{theorem}

So the fusion on $S$, a Sylow $p$-subgroup of a nilpotent group $G$, satisfies $\Ff_S(G)=\Ff_S(S)$.

\begin{definition}
An automorphism $\varphi$ of a group $G$ is said to be \emph{fixed-point-free} if it leaves
only the identity element of $G$ fixed.
\end{definition}

The following result shows that fixed-point-free are compatible with the $p$-local structure \cite[Theorem 10.1.2 and Lemma 10.1.3]{gorenstein-68}:

\begin{theorem}\label{thm:gor_10.1.2_3}
Let $G$ be a finite group and $p$ a prime dividing the order of $G$. If
$\varphi\colon G \to G$ is a fixed-point-free morphism then:
\begin{enumerate}[\rm (a)]
 \item There exists a unique $\varphi$-invariant Sylow $p$-subgroup $S$ of $G$,
and it contains every $\varphi$-invariant $p$-subgroup of $G$.
 \item If $H$ is a $\varphi$-invariant normal subgroup of $G$, then $\varphi$ induces
a fixed-point-free automorphism of $G/H$.
\end{enumerate}
\end{theorem}

Let $\Aut(\Ff)$ denote the group of automorphisms of $\Ff$:
\begin{multline*}
\Aut(\Ff)=\{\varphi\in\Aut(S)\mid\text{ if }\alpha \in  \Hom_\Ff(P,Q)\\ \text{ then }\varphi|_Q \circ \alpha \circ (\varphi|_P)^{-1}\in\Hom_\Ff(\varphi(P),\varphi(Q))\}
\end{multline*}
We are now ready to give a definition of what a fixed-point-free
automorphism of a saturated fusion system is.

\begin{definition}
\label{def:fixpointfreefusion}
Let $(S,\Ff)$ be a saturated fusion system. Then $\varphi \in \Aut(\Ff)$ is
a \emph{fixed-point-free} automorphism if the following hold
\begin{itemize}
 \item $\varphi\colon S \to S$ is fixed-point-free and
 \item $\varphi_\sharp\colon \Aut_\Ff(P) \to \Aut_\Ff(P)$, defined as
$\varphi_\sharp(\alpha)=\varphi \circ \alpha \circ (\varphi|_P)^{-1}$,
is fixed-point-free
for all $\varphi$-invariant subgroup $P \leq S$.
\end{itemize}
\end{definition}

The next result shows that this definition generalizes the concept of
fixed-point-free automorphism of a finite group:

\begin{proposition}\label{prop:fixedpointfree}
Let $G$ be a finite group and $p$ be a prime dividing the order of $G$.
If $\varphi\colon G \to G$ is a fixed-point-free automorphism then
$\varphi$ induces a fixed-point-free automorphism of $\Ff_G(S)$, where
$S$ is the only $\varphi$-invariant Sylow $p$-subgroup of $G$.
\end{proposition}
\begin{proof}
According to Theorem \ref{thm:gor_10.1.2_3}(a) there exists $S$ a unique $\varphi$-invariant Sylow $p$-subgroup of $G$.
As $\varphi|_S$ is the restriction of a fixed-point-free group morphism, it is
also fixed-point-free.

Consider now $P\leq S$ such that $\varphi(P)=P$. Then,
$C_G(P)$ and $N_G(P)$ are $\varphi$-invariant and $\varphi|_{C_G(P)}$
and $\varphi|_{N_G(P)}$ are fixed-point-free. Using
Theorem \ref{thm:gor_10.1.2_3}(b), $\varphi$ induces a fixed-point-free
group morphism $$\varphi_\sharp\colon \Aut_{\Ff_G(S)}(P)\cong N_G(P)/C_G(P) \to N_G(P)/C_G(P)\cong\Aut_{\Ff_G(S)}(P).$$
\end{proof}


\section{Thompson Theorems for saturated fusion systems}\label{Thompson}

In this section we give a unified proof of Theorems \ref{thm:A} and \ref{thm:B}. In order to it so, we introduce the concept of $\T$-automorphism:

\begin{definition}
Let $(S,\Ff)$ be a saturated fusion system. We say that \emph{$\Ff$ admits $\T$-automorphisms} if there exists $\U\leq\Aut(\Ff)$ such that one of the following holds:
\begin{itemize}
\item $\U=\langle\varphi\rangle$ where $\varphi$ is a fixed-point-free automorphism of prime order.
\item For every $\U$-invariant normal subgroup $Q\unlhd S$, $\Aut_\Ff(Q)$ is a $p$-group.
\end{itemize}
\end{definition}

The following lemma is a particular case for finite groups which can be proved directly.

\begin{lemma}\label{lemma:semidirect}
Let $G$ be the semidirect product $V\rtimes H$ where $V$ is an elementary abelian $p$-group $V$ and $H$ is
a group with an element of order prime to $p$ which does not centralize $V$. Then, given any $p$-sylow subgroup $S$ of $G$,
$\Ff_S(G)$ does not admit $\T$-automorphisms that leave $V$ invariant.
\end{lemma}
\begin{proof}
Assume $G=V\rtimes H$ is a minimal counterexample to the statement, $S$ be a $p$-sylow subgroup of $G$, and $\U$ be a group of $\T$-automorphisms of $\Ff_S(G)$ that leaves $V$ invariant. Since $V$ is $\U$-invariant normal in $G$, it is so in $S$: According to the hypothesis, $\Aut_\Ff(V)=N_G(V)/C_G(V)$ contains an element of order prime to $p$ (coming from $H$), hence we may assume that $\U=\langle\varphi\rangle$ where $\varphi$ is a fixed-point-free automorphism of prime order $r$. Without loss of generality, we may assume that $\varphi$ is an honest fixed-point-free automorphism of $G$.

The automorphism $\varphi$ restricts to a fixed-point-free automorphism of $V$, and according to
Theorem \ref{thm:gor_10.1.2_3}(b), $\varphi$ induces a fixed-point-free automorphism on $G/V\cong H$, namely $\widetilde{\varphi}$.

Let $q$ be any prime dividing the order of $H$. By Theorem \ref{thm:gor_10.1.2_3}(a), there exists $Q$, a $\widetilde{\varphi}$-invariant
Sylow $q$-subgroup in $H$. Now, if $Q\lneq H$,  $\varphi$ induces a fixed-point-free automorphism on $V\rtimes Q\lneq G$ in contradiction to the minimality of $G$. Therefore $H$ must be a $q$-group. Moreover, if $H$ is not abelian then the center $Z(Q)\lneq H$ is $\widetilde{\varphi}$-invariant and $\varphi$ induces a fixed-point-free automorphism on $V\rtimes Z(H)\lneq G$ in contradiction to the minimality of $G$. Hence $H$ is abelian.
Finally if $H$ is not elementary abelian, then characteristic subgroup $\Omega(H)\lneq H$ of elements of
order $q$ is $\widetilde{\varphi}$-invariant and $\varphi$ induces a fixed-point-free automorphism on $V\rtimes \Omega(H)\lneq G$ in contradiction to the minimality of $G$.

Now $r$, the order of $\varphi$, is different to $p$:
the restriction of $\varphi$ to $V$ gives a fixed-point-free action over $V$, a $p$-group,
so $r\neq p$.

We can assume that $H$ acts over $V$ without fixed points: consider
a set $A=\{ h_1, \dots , h_n\}$ of generators of $H$. Also $\varphi(A)$ generates $H$.
If $x \in V$ and $h\in H$, then we have
the identity $\varphi(h(x))=\varphi(h)(\varphi(x))$, so if $h(x)=x$ for all $h\in H$,
$\varphi(x)=\varphi(h)(\varphi(x))$. But $H=\{\varphi(h)\}_{h\in H}$ and the fixed points
by $H$ is a $\varphi$ invariant subgroup.
Let $N$ be the subgroup of $V$ of fixed points by $H$ and assume $N$ is not trivial.
Then we can construct the group $(V/N)\rtimes H$ and $\varphi$ induces a fixed-point-free
action by Theorem \ref{thm:gor_10.1.2_3}(b), and we get a contradiction with the minimality of $G$.

Consider now $L$ the semidirect product $H \rtimes \Z/r$.
The centralizer of $H$ in $L$ is itself, and we are
assuming that $V$ is a faithful $H$-module, and $p\neq r$. Then we can apply
\cite[Theorem 3.4.4]{gorenstein-68} to deduce that $\varphi$ cannot restrict to a fixed-point-free
automorphism of $V$, getting a contradiction.
\end{proof}

We now recall the definition of an $H$-free fusion system, for a finite group $H$ \cite[Section 6]{KL}: If $\Ff$ is a saturated fusion system over a $p$-group $S$ and $P$ is
an $\Ff$-centric fully normalized in $\Ff$ subgroup of $S$, then $N_\Ff(P)$ is a constrained fusion system \cite[Definition 4.1]{BCGLO} and
there is, up to isomorphism, a unique
finite group $L=L_P^{\Ff}$ having $N_S(P)$ as a Sylow $p$-subgroup such that
$C_L(P)=Z(P)$ and $N_\Ff(P)=\Ff_{N_{S}(P)}(L)$ \cite[Theorem 4.3]{BCGLO}. Then,
 
\begin{definition}\label{def:h-free}
Let $H$ is a finite group, and $\Ff$ be a saturated fusion system over a $p$-group $S$. We say that \emph{$\Ff$ is $H$-free}, if $H$ is not involved in any of the groups $L_P^{\Ff}$, with $Q$ running over the set of $\Ff$-centric fully normalized in $\Ff$ subgroups of $S$.
\end{definition}

\begin{proof}[Proof of Theorems \ref{thm:A} and \ref{thm:B}]
We will proceed considering a minimal counterexample and getting a contradiction.
So, let $S$ be the smallest $p$-group and $\Ff$ the saturated fusion system
with a minimal number of morphisms such that $(S,\Ff)$ admits $\U$ a group of $\T$-automorphisms and $\Ff_S(S)\lneq\Ff$.

\noindent\textbf{Step 1:} There exists a non-trivial elementary abelian proper
subgroup $W(S)\lneq S$ such that $(S,\Ff)=(S,N_\Ff(W(S)))$:

Given a group $G$, let $Z(G)$,
$J(G)$, and $\Omega(G)$ denote the center, the Thompson subgroup, and the group generated by the elements of order $p$ of $G$ respectively. Let $W(S)$ be the characteristic subgroup of $S$ defined in \cite[Section 4]{Onofrei-Stancu}. Then
$\Omega(Z(S))\leq W(S) \leq \Omega(Z(J(S)))$, and $W(S)$ is non-trivial and elementary abelian. Moreover, as $W(S)$ is characteristic,
it is $\U$-invariant and normal in $S$. This implies that $W(S)$ is fully normalized in
$\Ff$ so, by Proposition \ref{prop:fullynormalized}, $N_\Ff(W(S))$ is a saturated
fusion system over $S$.

Let us see now that $(S,\Ff)=(S,N_\Ff(W(S)))$: assume $N_\Ff(W(S))\lneq\Ff$.
Then $\U$ induces a group of $\T$-automorphisms in $(S,N_\Ff(W(S)))$, so if there are morphisms in $\Ff$ which are
not in $N_\Ff(W(S))$, by minimality of $\Ff$, $N_\Ff(W(S))=\Ff_S(S)$. But,
as $p$ is an odd prime (respectively $p=2$ and $\Ff$ is $\Sigma_4$-free),
by \cite[Theorem 1.3]{Onofrei-Stancu} (respectively \cite[Theorem 1.1]{Onofrei-Stancu})
this implies $\Ff=\Ff_S(S)$, so
$N_\Ff(W(S))=\Ff$, getting a contradiction.

If $W(S)=S$ then, applying Lemma \ref{lemma:semidirect}, $(S,\Ff)=(S,\Ff_S(S))$, so it is
not a counterexample.

\noindent\textbf{Step 2:} According to Lemma \ref{lemma:oliver-quotients}, $(S,\Ff)$ projects onto a saturated fusion system
$(S/W(S),\Ff/W(S))$ and $\U$ induces a group of $\T$-automorphism of $\Ff/W(S)$. So, by the minimality hypothesis, $\Ff/W(S)=\Ff_{S/W(S)}(S/W(S))$.

\noindent\textbf{Step 3:} There is an element $\alpha$ of prime order $q$ in $\Aut_\Ff(W(S))$,
with $q\neq p$:

As we are assuming that $\Ff_S(S)\lneq \Ff$, using the Alperin's Theorem
for saturated fusion systems (Theorem \ref{teo:alperin}), there exists $P$,
a fully normalized in $\Ff$,
$\Ff$-centric, $\Ff$-radical subgroup of $S$
with an $\Ff$-automorphism $\tilde \alpha$ of prime order $q$, $q\neq p$.
By \cite[Proposition 1.6]{BCGLO},
as $W(S)$ is normal in $\Ff$ and $P$ is $\Ff$-radical, $W(S)$ is contained in $P$
and we have induced maps
in the normal series $1 \trianglelefteq W(S) \trianglelefteq P$. By Step 2,
$\tilde \alpha$ projects to the identity in $P/W(S)$. If the restriction of
$\tilde \alpha$
to $W(S)$ is also the identity, by \cite[Theorem 5.3.2]{gorenstein-68}
the order of $\tilde \alpha$ is a power of $p$, getting a contradiction. So
$\tilde \alpha$ restricts to an automorphism $\alpha$ on $W(S)$ of order $q$.

\noindent\textbf{Last step:}
We finish the proof by considering $G=W(S) \rtimes H$ where $H=\Aut_\Ff(W(S))$, then the automorphism group induced
by $\U$ on $G$ is a group of $\T$-automorphism that leaves $W(S)$ invariant, what contradicts Lemma \ref{lemma:semidirect}. Therefore there is no minimal counterexample $(S,\Ff)$.
\end{proof}

We finish this section with a couple of examples that ilustrate the scope of these results. The first example, what seems to be well known for the experts (see comments below \cite[Theorem A]{Thompson1}), shows that the $\Sigma_4$-free hypothesis in the $p=2$ case is a necessary one.

\begin{example}\label{Ex:s4-free}
Let $G$ be the simple group $L(2,17)$, and let $S\leq G$ be a $2$-sylow subgroup. According to \cite[p.\ 9]{atlas}, the $2$-fusion is completely determined by its self normalizing sylow $S\cong D_{16}$ and two conjugacy classes of rank two elementary abelian subgroups of type $2A^2$ whose normalizer is isomorphic to $\Sigma_4$. So $\Ff=\Ff_S(G)$ is not $\Sigma_4$-free, $\Ff\ne\Ff_S(S)$, and for every normal $Q\unlhd S$, $\Aut_\Ff(Q)$ is a $2$-group (since the only maximal subgroup of $G$ of index coprime to $2$ is $S$). Therefore, for every $\U\leq\Aut(\Ff)$, and $\U$-invariant normal subgroup $Q\unlhd S$, $\Aut_\Ff(Q)$ is a $2$-group, but $\Ff\ne\Ff_S(S)$.
\end{example}

The last example shows that group automorphisms with fixed points can induce a fixed-point-free automorphisms of saturated fusion system.

\begin{example}\label{Ex:gp-vs-fusion}
Let $S$ be a finite $p$-group such that there exist a fixed-point-free automorphism $\phi$ of $S$. Consider $H=S\rtimes\langle\phi\rangle$ and let $N$ be a group whose order is coprime to $p$ and a homomorphism $f\colon H\to \text{Aut} (N)$. The homomorphism $f$ defines a semidirect product $G=N\rtimes H$, where the subgroup $K=SN$ is a normal subgroup of $G$ and $\phi$ acts on $K$ not necessary without fix points, but it acts on the fusion category $\Ff_S(G)=\Ff_S(S)$ without fix points.

An example of this situation is the group $A_4$ and $\phi\in\aut(A_4)$ given by conjugation in $S_4$ by the transposition $(1,2)$. Then $\phi$ fixes the Sylow $3$-subgroup $S=\langle (1,2,3)\rangle$, and furthermore it acts without fix points on the fusion category $\Ff_S(A_4)$.
\end{example}

\bibliographystyle{plain}
\bibliography{biblio}

\end{document}